\documentclass[12pt]{amsart}
 
\usepackage[dvips]{graphicx}
	\usepackage{amsmath,amssymb,amsthm,wrapfig,amsfonts,enumerate,latexsym}

\theoremstyle{plain}
  \newtheorem{thm}{Theorem}[section]

  \newtheorem{lem}[thm]{Lemma}
  \newtheorem{claim}[thm]{Claim}

  \newtheorem{conj}[thm]{Conjecture}

\theoremstyle{definition}
  \newtheorem{dfn}[thm]{Definition}
  \newtheorem{ex}[thm]{Example}

\theoremstyle{remark}
  \newtheorem{rem}[thm]{Remark}

\newcounter{yon}
\numberwithin{equation}{section}

	\newcommand{\dist}{\mathop{\mathit{d}} \nolimits}

	\newcommand{\tra}{\mathop{\mathrm{Tra}} \nolimits}
	\newcommand{\mushi}{\mathop{\mathrm{def}} \nolimits}
	\newcommand{\pr}{\mathop{\mathrm{proj}} \nolimits}

	\newcommand{\di}{\mathop{\mathrm{di}}   \nolimits}

	\newcommand{\sep}{\mathop{\mathrm{Sep}} \nolimits}
	\newcommand{\ler}{\mathop{\mathrm{LeRad}} \nolimits}

	\newcommand{\supp}{\mathop{\mathrm{Supp}}    \nolimits}

      \newcommand{\lm}{\mathop{\mathit{lm}} \nolimits}

    \newcommand{\cd}{\mathop{\mathrm{CD}}                  \nolimits}

    \newcommand{\ent}{\mathop{\mathrm{Ent}}
	\nolimits}
    
    \newcommand{\e}{\mathop{\varepsilon}       \nolimits}
    \DeclareMathOperator{\divv}{div}

\title[Estimates of eigenvalues of Laplacian]{Estimates of Eigenvalues
of Laplacian by a reduced number of subsets}
\author{Kei Funano}

\thanks{\hspace{-0.4cm}Supported by a Grant-in-Aid for
Scientific Research from the Japan Society for the Promotion of
Science.}

\address{Department of Mathematics, Faculty of Science, Kyoto University, Kyoto 606-8502, JAPAN}
\email{kfunano@math.kyoto-u.ac.jp}
\thanks{\hspace{-0.4cm}2000 Mathematics Subject Classification number: 58C40}
\date{}
\begin{document}
\maketitle
\begin{abstract}Chung-Grigor'yan-Yau's inequality describes upper bounds of eigenvalues of Laplacian
in terms of subsets (``input'') and their volumes.  In this paper we
 will show that we can reduce ginputh in Chung-Grigor'yan-Yau's inequality
 in the setting of Alexandrov spaces satisfying $\cd(0,\infty)$. We will
 also discuss a related conjecture for some universal inequality among eigenvalues of Laplacian.
 \end{abstract}
\section{Introduction}
The study of eigenvalues of Laplacian is now a classical but important subject in
mathematics. It is closely related with geometry of underlying spaces such as curvature,
volume, diameter, closed geodesics, and etc, see \cite{chavel}. In this
paper we prove the following. For two subsets $A,B$ in a metric space
$(X,\dist)$ we denote
\begin{align*}
\dist(A,B):=\inf \{ \dist(a,b) \mid a\in A,b\in B\}.
 \end{align*}
\begin{thm}\label{Mthm}There exists a universal and numerical constant $c>0$
 satisfying the following property.
 Let $(X,\mu)$ be an weighted compact finite-dimensional Alexandrov
 space satisfying $\cd (0,\infty)$ and $\mu(X)=1$. For any $l+1$ Borel subsets
 $A_0,A_1, \cdots, A_l$ with $l\leq k$, the $k$th eigenvalue $\lambda_k(X,\mu)$ of the
 weighted Laplacian has the estimate
\begin{align}
 \lambda_k(X,\mu)\leq \frac{c^{k-l+1}}{\min_{i\neq j}\dist(A_i,A_j)^2}\Big(\max_{i\neq j} \log \frac{1}{\mu(A_i)\mu(A_j)}\Big)^2.
 \end{align}
 \end{thm}

  Here an \emph{Alexandrov space} is a complete geodesic metric space
  with local 'sectional curvature' bounds introduced by A.~D.~Alexandrov
  in terms of comparison properties of geodesic triangles. The condition
  $\cd(0,\infty)$ stands for the space $(X,\mu)$ has nonnegative 'Ricci
  curvature' (see Section \ref{prfmain}).

 This inequality was first proved by Gromov and V.~Milman in the
 case where $k=l=1$ without curvature assumption (\cite{milgro}). It
 is equivalent to an exponential concentration inequality (see Lemma \ref{2.2l3}). Chung, Grigor'yan, and Yau generalized their result to the case where
$k=l$ (\cite{cgy1,cgy2}). Although Chung-Grigor'yan-Yau's setting was
 for manifolds, their proof also works for
 Alexandrov spaces without any changes. See Section \ref{weight} for details. The crucial point of the above theorem is that one can reduce
 the number of subsets (``input'') in a dimension-free way under
 assuming $\cd (0,\infty)$. 

\begin{rem}\upshape
 (1) The purpose in the previous paper \cite{funa2} was to use Theorem \ref{Mthm} to understand the relationships between the 
eigenvalues $\lambda_k (M,\mu)$ of the weighted Laplacian for different $k$, where $(M,\mu)$ is a
compact weighted Riemannian manifold $(M,\mu)$ having nonnegative Bakry-\'Emery Ricci curvature. Precisely, the
author obtained the universal inequalities $\lambda_k(M,\mu)\leq c^k
\lambda_1(M,\mu)$ among the
eigenvalues. After the paper \cite{funa2} was written Liu proved the
sharp universal inequality $\lambda_k(X,\mu)\leq c k^2 \lambda_1(X,\mu)$ for weighted compact
finite-dimensional Alexandrov space $(X,\mu)$ satisfying
$\cd(0,\infty)$ (\cite{Liu}). He pointed out that the so called the improved
Cheeger inequality holds via the same proof for graph setting in
\cite{kllot} and combining with the Buser-Ledoux inequality implies the
above sharp inequality.

 (2) In the previous paper \cite{funa2} the author proved Theorem \ref{Mthm}
 only for compact weighted manifolds. The proof in \cite{funa2}
 implicitly uses the smooth structure of the underlying spaces (see
 Section \ref{prfmain} for details). Theorem
 \ref{Mthm} in the present paper avoids the issue and is stronger than
 the one in \cite{funa2}. In fact, since our Alexandrov spaces are metric
 spaces with ``local'' sectional curvature bounds, our setting includes
 any compact weighted Riemannian manifolds satisfying $\cd (0,\infty)$. The
 author will not publish the paper \cite{funa2} from any journal.

 (3) One can extend Theorem \ref{Mthm} for spaces with lower negative Ricci
 curvature bounds $K$ if we add some restriction in diameter according to the lower
 bound $K$ and if we allow that the constant $c$ depends on diameter and
 $K$. See the proof of Theorem \ref{Mthm}. In Theorem \ref{Mthm} the lower bounds of Ricci
 curvature is necessary as was remarked in \cite{funa6}. Some 'dumbbell
 space' gives a counterexample, see \cite[Example 4.9]{funa6} for details.

 (4) Theorem \ref{Mthm} is true also in the case where $X$ has non-empty
 boundary. In that case we implicitly assume the Neumann
   boundary condition.
 \end{rem}

  In the proof of Theorem \ref{Mthm} we will
  work on the notion of \emph{'separation'}, which is regarded as a
  generalization of the concentration of measure phenomenon (see
  Subsection \ref{sepsec}). It tells the information whether or not
  there exists a pair which are not separated in some sense among any $k+1$-tuple subsets with a
  fixed volume. The idea of the proof of Theorem \ref{Mthm}
  will be discussed in Section \ref{prfmain} in details. 

\section{Preliminaries}\label{preliminaries}
We review some basics needed in this paper.
\subsection{L\'evy radius}


Let $X$ be an \emph{mm-space}, i.e.,
a complete separable metric space with a Borel probability measure $\mu_X$.



Let $f:X\to \mathbb{R}$ a Borel measurable
  function. A real number $m_f$ is called a \emph{median} of $f$
  if it satisfies that
  \begin{align*}
    \mu_X(\{ x\in X \mid f(x)\geq m_f\})\geq 1/2 \quad \text{and}\quad
    \mu_X(\{x\in X \mid f(x)\leq m_f\})\geq 1/2.
  \end{align*}The set of all median of the
  function $f$ is a bounded closed interval $[\,a_f,b_f\, ]$.
  We define $\lm(f;\mu_X):=(a_f+b_f)/2$.

 \begin{dfn}[L\'{e}vy radius]\label{ptd1}\upshape For $\kappa >0$, we define the \emph{L\'{e}vy
	radius} $\ler (X;- \kappa)$ of an mm-space $X$ as the infimum of $\rho >0$ such that
	every $1$-Lipschitz function $f:X\to \mathbb{R}$ satisfies that
    \begin{align*}
	\mu_X (\{ x\in X \mid |f(x)-\lm(f;\mu_X)|\geq \rho      \})\leq \kappa.
    \end{align*}
\end{dfn}

Refer to \cite{gromov}, \cite{ledoux} for the background of L\'evy radius. 

\subsection{Separation distance}\label{sepsec}We define the separation distance which
 plays an important role in the proof of Theorem \ref{Mthm}. The separation distance was introduced by Gromov in \cite{gromov}. 

\begin{dfn}[Separation distance]\label{2.2d1}
  For any $\kappa_0,\kappa_1,\cdots,\kappa_k\geq 0$ with $k\geq 1$,
  we define the ($k$-)\emph{separation distance}
  $\sep(X;\kappa_0,\kappa_1, \cdots, \kappa_k)$ of $X$
  as the supremum of $\min_{i\neq j}\dist(A_i,A_j)$,
  where $A_0,A_1, \cdots, A_k$ are
  any Borel subsets of $X$ satisfying that $\mu_X(A_i)\geq \kappa_i$ for
  all $i=0,1,\cdots,k$.
\end{dfn}It is immediate from the definition that if $\kappa_i\geq
  \tilde{\kappa_i}$ for each $i=0,1,\cdots,k$, then
  \begin{align*}
   \sep(X;\kappa_0,\kappa_1,\cdots,\kappa_k)\leq \sep(X;\tilde{\kappa}_0,\tilde{\kappa}_1,\cdots,\tilde{\kappa}_k).
   \end{align*}Note that if the support of $\mu_X$ is connected, then
 \begin{align*}\sep(X;\kappa_0, \kappa_1, \cdots, \kappa_k)=0
  \end{align*}for any $\kappa_0,\kappa_1,\cdots, \kappa_k>0$ such that
  $\sum_{i=0}^k \kappa_i>1$.

    We denote the closed $r$-neighborhood of a subset $A$ in
    a metric space by $C_r(A)$.
   \begin{lem}\label{2.2l2}Let $X$ be an mm-space and $k\geq 1$. Put $r:= \sep (X, \kappa_0,\kappa_1,\cdots,
    \kappa_{k})$. Assume that $k$ Borel subsets $A_0,A_1,\cdots,
    A_{k-1}$ of $X$
    satisfy $\mu_X(A_i)\geq \kappa_i$ for every $i=0,1,\cdots, k-1$ and
    $\dist(A_i,A_j)>r$ for every $i\neq j$. Then
    we have
    \begin{align*}
     \mu_X\Big(\bigcup_{i=0}^{k-1} C_{r} (A_i) \Big)\geq 1-\kappa_k.
     \end{align*}
    \begin{proof}
     Suppose that for some $\e_0>0$,
     \begin{align*}
      \mu_X\Big(\bigcup_{i=0}^{k-1} C_{r+\e_0} (A_i) \Big)\leq 1-\kappa_k.
      \end{align*}Putting $A_k:=X\setminus \bigcup_{i=0}^{k-1}
     C_{r+\e_0} (A_i)$ we have $\mu_X(A_k)\geq \kappa_k$ and $\dist(A_k,A_i)\geq r+\e_0$ for any
     $i=0,1,\cdots,k-1$. Thus we get
     \begin{align*}
      r<\min_{i\neq j} \dist(A_i,A_j)\leq
      \sep(X;\kappa_0,\kappa_1,\cdots, \kappa_k)=r,
      \end{align*}which is a contradiction. Hence $
     \mu_X(\bigcup_{i=0}^{k-1} C_{r+\e} (A_i) )> 1-\kappa_k$
     for any $\e>0$. Letting $\e\to 0$ we obtain the conclusion.
     \end{proof}
    \end{lem}

\begin{lem}\label{medul1}For any $\kappa>0$ we have
 \begin{align*}
            \ler (X;-\kappa)\leq \sep (X;\kappa/2,\kappa/2)
 \end{align*}
 \begin{proof}The lemma follows from a result in \cite{gromov}
  (cf.~\cite[Lemmas
  2.22, 2.24]{funa6}).
  \end{proof}
 \end{lem}

\subsection{Three distances between probability measures}
Let $X$ be a complete separable metric space. We denote by $\mathcal{P}(X)$
the set of Borel probability measures on $X$.

 \begin{dfn}[Prohorov distance]\label{2.3d1}
   Given two measures $\mu,\nu \in \mathcal{P}(X)$ and $\lambda \geq 0$,
   we define the \emph{Prohorov distance} $\di_{\lambda}(\mu,\nu)$
   as the infimum of $\varepsilon >0$ such that
  \begin{align}\label{2.3s1}
   \mu ({C}_{\varepsilon}(A))\geq
    \nu(A)-\lambda \varepsilon \text{ and }\nu({C}_{\varepsilon}(A))\geq
    \mu(A)-\lambda \varepsilon
  \end{align}
  for any Borel subsets $A\subseteq X$.
\end{dfn}

For any $\lambda \geq 0$, the function $\di_{\lambda}$ is a complete
separable distance function on $\mathcal{P}(X)$. If $\lambda >0$, then
the topology on $\mathcal{P}(X)$ determined by the Prohorov distance
function $\di_{\lambda}$ coincides with that of the weak convergence
(see \cite[Section 6]{bil}).
The distance functions $\di_{\lambda}$ for all $\lambda > 0$
are equivalent to each other. Also it is known that if $\mu
({C}_{\varepsilon}(A))\geq \nu(A)-\lambda \varepsilon$ for any Borel
subsets $A$ of $X$, then $\di_{\lambda}(\mu,\nu)\leq \e$. In other
words, the second inequality in (\ref{2.3s1}) follows from the first one
(see \cite[Section 6]{bil}).

For $(x,y)\in X\times X$, we put $\pr_1(x,y):=x$ and
$\pr_2(x,y):=y$. For two finite Borel measures $\mu$ and $\nu$ on $X$,
we write $\mu \leq \nu$ if $\mu(A)\leq \nu(A)$ for any Borel subset
$A\subseteq X$. A finite Borel measure $\pi$ on $X\times X$ is called
a \emph{partial transportation} from $\mu \in \mathcal{P}(X)$ to $\nu
\in \mathcal{P}(X)$ if $(\pr_1)_{\ast}(\pi)\leq \mu$ and
$(\pr_2)_{\ast}(\pi)\leq \nu$. Note that we do not assume $\pi$ to be
a probability measure. For a partial transportation $\pi$ from $\mu$
to $\nu$, we define its \emph{deficiency} $\mushi \pi$ by $\mushi \pi
:= 1-\pi(X\times X)$. Given $\varepsilon >0$, the partial
transportation $\pi$ is called an \emph{$\varepsilon$-transportation}
from $\mu$ to $\nu$ if it is supported in the subset
\begin{align*}
  \{ (x,y)\in X\times X \mid \dist(x,y)\leq
  \varepsilon\}.
\end{align*}

\begin{dfn}[Transportation distance]\label{2.3d2}
  Let $\lambda\geq 0$. For two probability measures $\mu,\nu \in
  \mathcal{P}(X)$, we define the \emph{transportation distance}
  $\tra_{\lambda}(\mu,\nu)$ between $\mu$ and $\nu$
  as the infimum of $\varepsilon >0$ such
  that there exists an $\varepsilon$-transportation $\pi$ from $\mu$
  to $\nu$ satisfying $\mushi \pi \leq \lambda \varepsilon$.
\end{dfn}

The following theorem is due to V.~Strassen.

\begin{thm}[{\cite[Corollary 1.28]{villani2},
    \cite[Section $3\frac{1}{2}.10$]{gromov}}]\label{2.3t3}
  For any $\lambda>0$, we have
  \[
  \tra_{\lambda} = \di_{\lambda}.
  \]
\end{thm}

Let $(X,\dist)$ be a complete metric space. We indicate by
$\mathcal{P}^2(X)$ the set of all Borel probability measures $\nu \in
\mathcal{P}(X)$ such that
\[
\int_{X}\dist(x,y)^2 d\nu (y)<+\infty
\]
for some $x\in X$.

\begin{dfn}[($L^2$-)Wasserstein distance]\label{2.3d4}
  For two probability measures $\mu,\nu \in \mathcal{P}^2(X)$, we
  define the \emph{$L^2$-Wasserstein distance $\dist_2^W(\mu,\nu)$}
  between $\mu$ and $\nu$ as the infimum of
  \begin{align*}
    \Big(\int_{X\times X}\dist(x,y)^2  d\pi
    (x,y)\Big)^{1/2},
  \end{align*}
  where $\pi \in \mathcal{P}^2(X\times X)$ runs over all \emph{couplings}
  of $\mu$ and $\nu$,
  i.e., probability measures $\pi$ with the property that $\pi (A\times
  X)=\mu(A)$ and $\pi (X\times A)=\nu (A)$ for any Borel subset
  $A\subseteq X$. It is known that this infimum is achieved by some
 transport plan, which we call an \emph{optimal transport plan} for $\dist_2^W(\mu,\nu)$. 
\end{dfn}

If the underlying space $X$ is compact, then the topology on
$\mathcal{P}(X)$ induced from the $L^2$-Wasserstein distance function
coincides with that of the weak convergence (see \cite[Theorem 7.12]{villani2}).

\subsection{Weighted Alexandrov spaces}\label{weight}

We refer to \cite{bbi,bgp} for basics of Alexandrov spaces and to \cite[Section 4]{Kuwae1}, \cite{Kuwae} for analysis of
Alexandrov spaces.

Let $X$ be a compact $n$-dimensional Alexandrov space and
$\mathcal{H}^n$ be its Hausdorff measure.
Let $\mu$ be a probability measure on $X$ defined
by $d\mu := e^{-V}d\mathcal{H}^n$, where $V$ is a function on $X$ with a
certain regularity condition (e.g., any Lipschitz continuous function is sufficient
in the following argument). For the measure $\mu$ we define the \emph{weighted Laplacian} (also called as the
\emph{Witten Laplacian}) $\Delta_{\mu}$ by
\begin{align*}
 \Delta_{\mu}:=\Delta + \nabla V\cdot \nabla  =-e^{-V} \divv (e^{-V}\nabla \cdot),
\end{align*}where $\Delta$ is the nonnegative Laplacian. $\Delta_{\mu}$
has discrete spectrum consisting of eigenvalues
\begin{align*}
 0=\lambda_0(X,\mu)<\lambda_1(X,\mu)\leq \cdots \leq \lambda_k(X,\mu)\leq
 \cdots. 
 \end{align*}

 We remark that Chung-Grigor'yan-Yau's theorem (the case where $k=l$ in
     Theorem \ref{Mthm}) holds for weighted compact finite-dimensional Alexandrov
     spaces. In fact, in the proof of the theorem we need only the
     Davies-Gaffney (weighted) heat kernel
     estimate
     \begin{align*}
      \int_A\int_B p_t(x,y)d\mu(x)d\mu(y)\leq \sqrt{\mu(A)\mu(B)}\exp\Big(-\frac{\dist^2(A,B)}{4t}\Big)
      \end{align*}for any Borel subsets $A,B$ and asymptotic expansion
      of (weighted) heat kernel by eigenvalues and
     eigenfunctions of Laplacian (\cite{cgy1}). These are true for
     weighted compact finite-dimensional Alexandrov spaces (\cite{sturm4},
     \cite{Kuwae}).

\section{Proof of Theorem \ref{Mthm}}\label{prfmain}

       In order to prove Theorem \ref{Mthm} we need to explain some
       useful tools from the theory of optimal transportation. Refer to \cite{villani2,villani}
       for more details.

Let $(X,\dist)$ be a metric space. A rectifiable curve
$\gamma:[\, 0,1\,]\to X$ is called a \emph{geodesic} if its arclength
coincides with the distance $\dist(\gamma(0),\gamma(1))$ and it has a
constant speed, i.e., parameterized proportionally to the arclength. We
say that a metric space is a \emph{geodesic space}
if any two points are joined by a geodesic between them. It is known
that $(\mathcal{P}^2(X),\dist_2^W)$ is a compact geodesic space as soon as
$X$ is (\cite[Proposition 2.10]{sturm}). 


Let $X$ be a finite-dimensional Alexandrov space. For two probability measures $\mu_0,\mu_1\in \mathcal{P}^2(M)$ which are
absolutely continuous with respect to the Hausdorff measure, there is a unique geodesic $(
\mu_t)_{t\in [\,0,1\,]}$ between them with respect to the $L^2$-Wasserstein
distance function $\dist_2^W$ (\cite{mcc}, \cite[Theorem 1.1]{ber}).

 For an mm-space $X$ let us denote by $\Gamma$ the set of minimal
 geodesics $\gamma:[\,0,1\,]\to X$ endowed with the distance
 \begin{align*}
  \dist_{\Gamma}(\gamma_1,\gamma_2):=\sup_{t\in [\,0,1\,]}\dist(\gamma_1(t),\gamma_2(t)).
  \end{align*}Define the \emph{evaluation map} $e_t:\Gamma \to X$ for
  $t\in [\,0,1\,]$ as $e_t(\gamma):=\gamma(t)$. A probability measure
  $\Pi\in \mathcal{P}(\Gamma)$ is called a \emph{dynamical optimal
  transference plan} if the curve $\mu_t:=(e_t)_{\ast}\Pi$, $t\in [\,0,1\,]$, is a
  geodesic in $(\mathcal{P}^2(X),\dist_2^W)$. Then
  $\pi:=(e_0\times e_1)_{\ast}\Pi$ is an optimal coupling of $\mu_0$ and
  $\mu_1$, where $e_0\times e_1:\Gamma\to X\times X$ is the ``endpoints''
  map, i.e., $(e_0\times e_1)(\gamma):=(e_0(\gamma),e_1(\gamma))$.

  \begin{lem}[{\cite[Proposition 2.10]{lott-villani}}]\label{3l1}If
  $(X,\dist)$ is locally compact, then any geodesic
   $(\mu_t)_{t\in [\,0,1\,]}$ in
  $(\mathcal{P}^2(X),\dist_2^W)$ is associated with a dynamical optimal transference
  plan $\Pi$, i.e., $\mu_t=(e_t)_{\ast}\Pi$.
   \end{lem}

 Let $\mu$ and $\nu$ be two probability measures on a set $X$. We define
 the \emph{relative entropy} $\ent_{\mu}(\nu)$ of $\nu$ with respect to
 $\mu$ as follows. If $\nu$ is absolutely continuous with respect to $\mu$,
 writing $d\nu=\rho d\mu$, then
 \begin{align*}
  \ent_{\mu} (\nu):=\int_X \rho \log \rho   d \mu,
  \end{align*}otherwise $\ent_{\mu}(\nu):=\infty$.

  \begin{dfn}[{Curvature-dimension condition, \cite{lott-villani}, \cite{sturm3,sturm}}]\label{3d1}Let $K$ be a real number. We say that
   a locally compact mm-space $X$ satisfies the \emph{curvature-dimension condition} $CD(K,\infty)$
   if for any $\nu_0,\nu_1\in \mathcal{P}^2(X)$ there exists a minimal
   geodesic $( \nu_t )_{t\in [\,0,1\,]}$ in $(\mathcal{P}^2(X),\dist_2^W)$
   from $\nu_0$ to $\nu_1$ such that
   \begin{align*}
    \ent_{\mu_X}(\nu_t)\leq (1-t)\ent_{\mu_X}(\nu_0)+t
   \ent_{\mu_X}(\nu_1)-\frac{K}{2}(1-t)t \dist_2^W(\nu_0,\nu_1)^2
    \end{align*}for any $t\in [\,0,1\,]$.
   \end{dfn}
  \begin{ex}
   \begin{enumerate}
    \item
 A complete weighted Riemannian manifold $(M,\mu)$ has Bakry-\'Emery
         Ricci curvature $\geq K$ for some $K\in \mathbb{R}$ if and only if
  $(M,\mu)$ satisfies $\cd (K,\infty)$ (\cite{cms1,cms2}, \cite{vrs},
         \cite{sturm2}).
     \item An $n$-dimensional Alexandrov space of curvature $\geq K$
           satisfies $\cd ((n-1)K,\infty)$ (\cite{pet}, \cite{zhang-zhu}).
    \end{enumerate}
  \end{ex}

In the above definition, assume that both $\nu_0$ and $\nu_1$ are
absolutely continuous with respect to $\mu_X$. Then Jensen's inequality applied to the convex function $r\mapsto r\log r$ gives
\begin{align}\label{3s1}
     & \log \mu_X(\supp \nu_t)\\ \geq \ & -(1-t) \int_{X} \rho_0 \log
 \rho_0 d\mu_X -t\int_X  \rho_1 \log \rho_1 d\mu_X + \frac{Kt(1-t)}{2}
 \dist_2^W(\nu_0,\nu_1)^2, \notag
\end{align}where $\rho_0$ and $\rho_1$ are densities of $\nu_0$
 and $\nu_1$ with respect to $\mu_X$ respectively.
 In particular, for two Borel subsets $A,B\subseteq X$ with
 $\mu_X(A),\mu_X(B)>0$, we have
 \begin{align}\label{3s1cont}
  &\log \mu_X (\supp \nu_t)\\ \geq \ &(1-t)\log \mu_X(A)+t\log
  \mu_X(B)+\frac{Kt(1-t)}{2}\dist_2^W\Big(\frac{\mu_X|_A}{\mu_X(A)},
  \frac{\mu_X|_B}{\mu_X(B)} \Big)^2 \tag*{}
  \end{align}(\cite{sturm}).

Theorem \ref{Mthm} follows from the following key theorem together with
Chung-Grigor'yan-Yau's theorem.

\begin{thm}\label{3t2}
 Let $(X,\mu)$ be an weighted finite-dimensional Alexandrov space
 satisfying $\cd(0,\infty)$ and $k\geq 2$. If $(X,\mu)$
 satisfies
 \begin{align}\label{3ss1}
  \sep((X,\mu);\underbrace{\kappa,\kappa, \cdots,\kappa}_{k+1\text{ times}})\leq
  \frac{1}{D} \log \frac{1}{\kappa}
  \end{align}for any $\kappa>0$, then we have
 \begin{align}\label{conc1}
  \sep((X,\mu);\underbrace{\kappa,\kappa, \cdots,\kappa}_{k\text{ times}})\leq
  \frac{c}{D} \log \frac{1}{\kappa}
  \end{align}for any $\kappa>0$ and for some universal numeric constant
 $c>0$.
 \end{thm}

 \begin{proof}[Proof of Theorem \ref{Mthm}]In the assumption
  (\ref{3ss1}) of Theorem \ref{3t2} we can take $D$ as some universal constant times $\sqrt{\lambda_k(X,\mu)}$
  by Chung-Grigor'yan-Yau's
  theorem. Iterating Theorem \ref{3t2} $k-l$
  times we get
  \begin{align*}\sep
   ((X,\mu);\underbrace{\kappa,\kappa, \cdots,\kappa}_{l \text{
   times}})\leq \frac{c^{k-l+1}}{\sqrt{\lambda_k(X,\mu)}}\log\frac{1}{\kappa}
   \end{align*}for any $\kappa>0$ and for some universal constant
  $c>0$. Since
  \begin{align*}
  \sep((X,\mu);\kappa_0,\kappa_1,\cdots,\kappa_l)\leq \sep
   ((X,\mu);\underbrace{\min_{i}\kappa_i,\min_i \kappa_i, \cdots,\min_i \kappa_i}_{l \text{
   times}})
   \end{align*}for any $\kappa_0,\kappa_1,\cdots,\kappa_l>0$, we thereby obtain
  \begin{align*}\sep((X,\mu);\kappa_0,\kappa_1,\cdots,\kappa_l)\leq \frac{c^{k-l+1}}{\sqrt{\lambda_k(X,\mu)}}\max_i \log
   \frac{1}{\kappa_i}\leq
   \frac{c^{k-l+1}}{\sqrt{\lambda_k(X,\mu)}}\max_{i\neq j} \log
   \frac{1}{\kappa_i\kappa_j}.
   \end{align*}This completes the proof of Theorem \ref{Mthm}.
  \end{proof}
 The rough idea of the proof of Theorem \ref{3t2} in
 \cite{funa2} for smooth manifolds was the following. It turns out that it is enough to prove (\ref{conc1}) for sufficiently
 small $\kappa>0$ and sufficiently large $c>0$. We suppose the converse
 of this, i.e.,
 \begin{align*}
  \sep ((X,\mu);\underbrace{\kappa,\kappa, \cdots,\kappa}_{k\text{ times}})> \frac{c}{D}\log \frac{1}{\kappa}
  \end{align*}for sufficiently small $\kappa>0$ and sufficiently large
  $c>0$. Put $\alpha:= (c/D)\log (1/\kappa)$. By the definition of the
  separation distance there exists $k$
  Borel subsets $A_0,A_1.\cdots, A_{k-1}\subseteq M$ such that
  $\min_{i\neq j}\dist(A_i,A_j)>\alpha$ and $\mu(A_i)\geq \kappa$ for
  any $i$. If we choose $c$ is greater than $400$, then by the
  assumption (\ref{3ss1}) we have
  \begin{align*}
   \sep ((X,\mu);\underbrace{\kappa,\kappa, \cdots,\kappa}_{k\text{
   times}}, \kappa^{100})\leq \sep
   ((X,\mu);\underbrace{\kappa^{100},\kappa^{100},
   \cdots,\kappa^{100}}_{k+1\text{ times}})\leq \frac{100}{D}\log
   \frac{1}{\kappa}\leq \frac{\alpha}{4}.
   \end{align*}Lemma \ref{2.2l2} implies
   \begin{align*}
    \mu \Big(\bigcup_{i=0}^{k-1} C_{\alpha/4}(A_i) \Big)\geq 1-\kappa^{100}.
    \end{align*}It means that if $\kappa>0$ is sufficiently small, the measure
    of the set
    $\bigcup_{i=0}^{k-1} C_{\alpha/4}(A_i)$ is nearly $1$. Although it is not true, we assume that
    \begin{align}\label{uso}
     \mu \Big(\bigcup_{i=0}^{k-1} C_{\alpha/4}(A_i) \Big)=1
     \end{align}in order to tell the idea of the proof. Putting $A:=C_{\alpha/4}(A_0)$ and
     $B:=\bigcup_{i=1}^{k-1} C_{\alpha/4}(A_i)$, we have $X=A\cup B$,
     $A\cap B=\emptyset$, $\mu(A)\geq \kappa,\mu(B)\geq \kappa$, and $\dist(A,B)\geq \alpha/2$. 

     Let $(\mu_t)_{t\in[\,0,1\,]}$ be a geodesic from
   $\mu_A:=(1/\mu(A))\mu|_A$ to $\mu$ with respect to $\dist_2^W$. For
   sufficiently small $t>0$ we have $\dist(x,A)<\alpha/2 \leq
   \dist(A,B)$ for any $x\in \supp \mu_t$, which gives $\supp
   \mu_t\subseteq A$. This leads a contradiction since by (\ref{3s1cont}) we have
   \begin{align}\label{israel}
    \log \mu(A)\geq \log \mu(\supp \mu_t)\geq (1-t)\log \mu(A)+t\log \mu(X),
    \end{align}which implies $\log \mu(A)\geq 0$. Since (\ref{uso}) is
    always not true, we have an error term depending only on
    $\kappa$ in (\ref{israel}) and we need to consider the trade-off
    between the error term and $t$ to accomplish the above idea. This leads us to control separated
    subsets and estimate transport distances between them. In \cite{funa2}, in
    order to control separated subsets the
    author heavily relied on E.~Milman's theorem in \cite{emil3}(see
    \cite[Claim 3.5]{funa2}). His theorem is not known
    for singular metric spaces such as Alexandrov spaces. The key
    ingredient of his theorem relies on the regularity theory of
    isoperimetric minimizer. Below we will avoid using his theorem. From $A_i$ we
    will construct two subsets $A,B$ such that the transport distance
    between them is at most $c\dist(A,B)$. The union of $A,B$ does not
    necessarily have almost total measure.

 \begin{proof}[Proof of Theorem \ref{3t2}]
  It suffices to prove that there exist two universal
  numeric constants $c_0,\kappa_0>0$ such that
  \begin{align}\label{3s2}
   \sep ((X,\mu);\underbrace{\kappa,\kappa, \cdots,\kappa}_{k\text{ times}})\leq \frac{c_0}{D}\log \frac{1}{\kappa}
   \end{align}for any $\kappa\leq \kappa_0$. In fact, if $\kappa\geq
  1/2$, then the left-hand side of the above inequality is zero and there is nothing to prove. In the
  case where $\kappa_0< \kappa \leq 1/2$, by (\ref{3s2}) we have 
  \begin{align*}
   \sep((X, \mu);\underbrace{\kappa,\kappa, \cdots,\kappa}_{k\text{
   times}})\leq \ &
   \sep((X,\mu);\underbrace{\kappa_0,\kappa_0,\cdots,\kappa_0}_{k\text{
   times}})\\
   \leq \ &\frac{c_0\log
   \frac{1}{\kappa_0}}{D\log \frac{1}{\kappa}} \log
   \frac{1}{\kappa}\\ \leq \ & \frac{c_0\log
   \frac{1}{\kappa_0}}{D\log 2} \log \frac{1}{\kappa},
   \end{align*}which implies the conclusion of the theorem.

  Suppose the contrary to (\ref{3s2}), i.e.,
  \begin{align}\label{3s3}
   \sep((X,\mu);\underbrace{\kappa,\kappa, \cdots,\kappa}_{k\text{ times}}) > \frac{c_1}{D} \log \frac{1}{\kappa},
   \end{align}where $c_1>0$ is a sufficiently large universal numeric
  constant and $\kappa>0$ is a sufficiently small number. Both the largeness of $c_1$ and the
  smallness of $\kappa$ will be specified later. Note that the assumption (\ref{3s3}) immediately gives $k
  \kappa<1$ (otherwise, the left-hand side of (\ref{3s3}) is zero). We denote
  the right-hand side of (\ref{3s3}) by $\alpha$, i.e.,
  \begin{align*}
   \alpha:=\frac{c_1}{D}\log \frac{1}{\kappa}.
   \end{align*}

  The assumption (\ref{3s3}) implies the existence of $k$ Borel subsets
  $A_0,A_1,\cdots, A_{k-1}\subseteq X$ such that $\mu(A_i)\geq \kappa$
  for any $i$ and $\dist(A_i,A_j)> \alpha \text{ \ for any\ }i\neq j$.
   If $c_1$ is large enough, then after applying Lemma \ref{2.2l2} to the
  condition (\ref{3ss1}) we may assume that those $A_0,A_1, \cdots,
  A_{k-1}$ are compact subsets and satisfy
  $\mu(A_i)\geq \kappa$, $\dist(A_i,A_j)\geq \alpha/2$, and
  $\mu(\bigcup_{i=0}^{k-1}A_i)\geq 1-\kappa^8$. In fact, the
  assumption (\ref{3ss1}) yields
  \begin{align*}
   \sep((X,\mu);\underbrace{\kappa,\kappa, \cdots,\kappa}_{k\text{
   times}},\kappa^9)\leq \sep ((X,\mu);\underbrace{\kappa^9,\kappa^9, \cdots,\kappa^9}_{k+1\text{
   times}})\leq \frac{9}{D} \log \frac{1}{\kappa}.
   \end{align*}Thus whenever $c_1>36$ we get
  \begin{align*}\sep((X,\mu);\underbrace{\kappa,\kappa, \cdots,\kappa}_{k\text{
   times}},\kappa^9)<\alpha/4,
   \end{align*}which implies $\mu(\bigcup_{i=0}^{k-1}
  C_{\alpha/4}(A_i))\geq 1-\kappa^9$ by Lemma \ref{2.2l2}. Note that
  $\mu (C_{\alpha/4}(A_i))>\mu(A_i)\geq \kappa$ since $\mu$ has
  full support on $X$ . We can
  approximate $C_{\alpha/4}(A_i)$ by compact subsets $K_i\subseteq
  C_{\alpha/4}(A_i)$ so that $\mu(K_i)$ is as close as possible to
  $\mu(C_{\alpha/4}(A_i))$ (\cite[Theorems 1.1 and 1.3]{bil}). After
  taking a sufficient approximation $K_i$ we rechoose $A_i$ as $K_i$.

  For each $i$ we set $\mu_{A_i}:=(1/\mu(A_i))\mu|_{A_i}$. Given any $i,j$ we take a $1$-Lipschitz function $f_{ij}:X\to
  \mathbb{R}$ such that
  \begin{align*}
   |\lm(f_{ij};\mu_{A_i})-\lm(f_{ij};\mu_{A_j})|=\sup
   |\lm(f;\mu_{A_i})-\lm (f;\mu_{A_j})|,
   \end{align*}where the supremum runs over all $1$-Lipschitz functions
  $f:X\to \mathbb{R}$. We can take such $f_{ij}$ so that $f_{ij}=f_{ji}$.

  Put $a_i=a_i(\kappa):=\ler ((A_i,\mu_{A_i});-\kappa^8)$.
  \begin{claim}\label{madacl1}
   We have $a_i\leq \alpha/2$ provided that $c_1$ is large enough.
   \begin{proof}By Lemma \ref{medul1} we have $a_i\leq
    \sep(\mu_{A_i};\kappa^8/2,\kappa^8/2)$. We shall estimate the
    right-hand side. For any $B,C\subseteq A_i$ such that
    $\mu_{A_i}(B),\mu_{A_i}(C)\geq \kappa^8/2$ we have
    $\mu(B),\mu(C)\geq \mu(A_i)\kappa^8/2\geq \kappa^9/2$. If $c_1$ is
    large enough, then by the assumption (\ref{3ss1}) we get
    \begin{align*}
     \sep\Big((X,\mu);\underbrace{\kappa,\kappa, \cdots,\kappa}_{k-1\text{ times}},\frac{\kappa^9}{2},\frac{\kappa^9}{2}\Big)\leq \frac{1}{D}\log\frac{2}{\kappa^9}<\frac{\alpha}{2}.
     \end{align*}Since $\dist(A_j, A_k)\geq \alpha/2$ we hence obtain
    $\dist(B,C)<\alpha/2$, which implies the claim.
    \end{proof}
   \end{claim}Setting
  \begin{align*}
   A_i':= \bigcap_{j=0}^{k-1}\{ x\in A_i \mid |f_{ij}(x)-\lm(f_{ij};\mu_{A_i})|\leq a_i\},
   \end{align*}we have $\mu_{A_i}(A_i')\geq 1-\kappa^8 k$ by the definition of
  L\'evy radius. Recalling that $\kappa<1/k$ we get
  $\mu_{A_i}(A_i')\geq 1-\kappa^7$. We also get
  $\mu(\bigcup_{i=0}^{k-1}A_i')=
  \sum_{i=0}^{k-1}\mu(A_i)\mu_{A_i}(A_i')\geq 1-\kappa^8 -\kappa^7\geq
  1-\kappa^6$ provided that $\kappa$ is small enough. 

 Take $x_i\in A_i'$ and $x_j\in A_j'$ such that
  $\dist(x_i,x_j)=\dist(A_i',A_j')$. Claim \ref{madacl1} yields that for any $1$-Lipschitz
  functions $f:X\to \mathbb{R}$ we have
  \begin{align}\label{madala1}
   |\lm(f;\mu_{A_i})-\lm(f;\mu_{A_j})|\leq \ &
   |\lm(f_{ij};\mu_{A_i})-\lm(f_{ij};\mu_{A_j})|\\
   \leq \ & a_i + a_j + |f_{ij}(x_i)-f_{ij}(x_j)|\tag*{}\\
   \leq \ & a_i+ a_j +\dist(A_i',A_j')\tag*{}\\
   \leq \ & 3\dist(A_i',A_j').\tag*{}
   \end{align}


  Without loss of generality one may assume that
  $\dist(A_0',A_1')=\min_{i\neq j} \dist(A_i',A_j')$ and $\mu(A_0')\leq
  \mu(A_1')$. Set $B_0:=A_0'$ and $B_1:=A_0'\cup A_1'$.
  \begin{claim}\label{3cl2}There exists a coupling $\pi$ of $\mu_{B_0}$ and $\mu_{B_1}$ such that
   \begin{align*}
    \pi (\{ (x,y)\in X\times X \mid \dist(x,y)\leq 8 \dist(A_0',A_1') \})\geq 1-\kappa^6.
    \end{align*}
   \begin{proof}Put $\delta:=8\dist(A_0',A_1')$. It suffices to prove
    that $\di_{\kappa^6/\delta}(\mu_{B_0},\mu_{B_1})\leq
    \delta$ according to Theorem \ref{2.3t3}.  In fact, Theorem \ref{2.3t3} gives that there exists a $\delta$-transportation $\pi_0$ from
  $\mu_{B_0}$ to $\mu_{B_1}$ such that $\mushi \pi_0 \leq \kappa^6$. If $\mushi \pi_0 =0$, then we set $\pi:=\pi_0$. If $\mushi
  \pi_0 >0$, then set
  \begin{align*}
   \pi:= \pi_0+\frac{1}{\mushi \pi_0 } (\mu_{B_0} - (\pr_1)_{\ast} \pi_0) \times
   (\mu_{B_1}- (\pr_2)_{\ast}\pi_0 ).
   \end{align*}It is easy to check that $\pi$ fulfills the desired
  property.

    Given a Borel subset $A\subseteq B_1$ we shall prove that
    \begin{align}\label{str1}
     \mu_{B_1}(C_{\delta}(A))\geq \mu_{B_0}(A)-\kappa^6.
     \end{align}As we remarked just after Definition \ref{2.3d1}, this
    implies the other inequality $\mu_{B_0}(C_{\delta}(A))\geq
    \mu_{B_1}(A)-\kappa^6$ and hence $\di_{\kappa^6/\delta}(\mu_{B_0},\mu_{B_1})\leq
    \delta$.

    To prove (\ref{str1}) we may assume that $\mu_{B_0}(A)\geq \kappa^6$. Define $f:X\to \mathbb{R}$ by $f(x):=\dist(x,C_{\delta/2}(A)\cap
    A_0')$. Note that
    \begin{align}\label{fuku1}
     \mu_{A_0}(A)\geq (\mu(B_0)/\mu(A_0)) \mu_{B_0}(A)\geq (1-\kappa^7)\kappa^6.
     \end{align}As we showed in the proof of Claim \ref{madacl1} we get
    \begin{align*}
     \sep(\mu_{A_0};\kappa^6(1-\kappa^7),\kappa^7)\leq
     \sep(\mu_{A_0};\kappa^7,\kappa^7)<\alpha< \delta/2,
     \end{align*}which gives $\mu_{A_0}(C_{\delta/2}(A))\geq 1-\kappa^7$
    by (\ref{fuku1}).
Since $\mu_{A_0}(C_{\delta/2}(A)\cap A_0')\geq
    \mu_{A_0}(C_{\delta/2}(A))-\kappa^7\geq 1-2\kappa^7> 1/2$ for
    sufficiently small $\kappa$ we have $\lm
    (f;\mu_{A_0})=0$. Using (\ref{madala1}) we then obtain
    \begin{align*}
     \lm (f;\mu_{A_1})\leq 3\dist(A_0',A_1').
     \end{align*}

    Put $B:= \{ x\in X \mid \dist(x,C_{\delta/2}(A)\cap A_0')\leq
     \delta/2 \}$. Then
    \begin{align*}
     \mu_{A_1}(B)\geq \ &\mu_{A_1}(\{ x\in X \mid f(x)\leq
     3\dist(A_0',A_1')+a_1\})\\
     \geq \ & \mu_{A_1}(\{ x\in X \mid
     |f(x)-\lm(f;\mu_{A_1})|\leq a_1\})
     \geq  1-\kappa^8,
     \end{align*}which shows
    \begin{align*}
    \mu(B\cap A_1')\geq \mu(B\cap A_1)-\mu(A_1\setminus A_1')\geq
     (1-\kappa^8)\mu(A_1)-\kappa^7 \mu(A_1)\geq
     (1-2\kappa^7)\mu(A_1).
     \end{align*}Combining this inequality with $\mu (B\cap A_0' )\geq
    \mu(C_{\delta}(A)\cap A_0')\geq (1-2\kappa^7)\mu(A_0)$ we obtain
\begin{align*}
    \mu_{B_1}(B)\geq  \mu(B_1)^{-1}\{ \mu(A_0')(1-2\kappa^7)+\mu(A_1')(1-2\kappa^7)
 \} \geq  1-4\kappa^7
 \geq  \mu_{B_0}(A)-\kappa^6 
 \end{align*}provided that $\kappa$ is sufficiently small. This
    completes the proof of the claim.
    \end{proof}
   \end{claim}

We set $\Delta:= \{ (x,y)\in X\times X  \mid
   \dist(x,y)\leq 8 \dist(A_0',A_1')  \}$. We consider two Borel probability measures $\mu_{0}:=a
  (\pr_1)_{\ast}(\pi |_{\Delta})$ and $\mu_{1}:=a
  (\pr_2)_{\ast}(\pi |_{\Delta})$, where $a:=
  \pi(\Delta)^{-1}$. By Claim \ref{3cl2} we have
  \begin{align}\label{3s11}
   1\leq a \leq 1/(1-\kappa^6)
   \end{align}and
  \begin{align}\label{3s12}
   \dist_2^W(\mu_{0},\mu_{1})^2 \leq a  \int_{X \times
   X }\dist (x,y)^2 d \pi |_{\Delta}(x,y) \leq 8^2
   \dist(B_0,B_1\setminus B_0)^2.
   \end{align}Take an optimal dynamical transference
  plan $\Pi$ such that $(e_i)_{\ast}\Pi =\mu_i$ for each
  $i=0,1$. Putting $r:=\dist (B_0,B_1\setminus
  B_0)$,
  we consider the set
  \begin{align*}
   \Gamma_t := \{   \gamma \in \supp \Pi  \mid \dist(e_0(\gamma),
   e_t(\gamma))\leq r/2                        \}.
   \end{align*}By (\ref{3s12}) we have
  \begin{align*}
   (r^2/4)\Pi (\Gamma\setminus \Gamma_t)\leq
   \dist_2^W((e_0)_{\ast}\Pi,(e_t)_{\ast}\Pi)^2=
   t^2\dist_2^W(\mu_0,\mu_1)^2 \leq \{8 t \dist(B_0,B_1\setminus B_0)\}^2,
   \end{align*}which yields
  \begin{align}\label{3s13}
   \Pi(\Gamma_t)\geq 1- c t^2,
   \end{align}where $c:=8^3$. For $s\in [\,0,1\,]$ we put $\nu_s := (e_s)_{\ast}
  \frac{\Pi|_{\Gamma_t}}{\Pi(\Gamma_t)}$. By the definition of $\nu_s$ we obtain the following. 
  \begin{claim}\label{3cl3} $\supp \nu_t
    \cap \bigcup_{i=0}^{k-1} A_i' \subseteq B_0$.
   \end{claim}By using Claim \ref{3cl3}, we get
  \begin{align}\label{3s15}
   \log \mu (B_0) + \frac{\kappa^6}{\mu (B_0)}
   \geq \ & \log \mu (B_0)+ \log
   \Big(1+\frac{\kappa^6}{\mu (B_0)}\Big) \\
   = \ &\log (\mu (B_0) +\kappa^6) \tag*{}\\
   \geq \ &\log \Big\{\mu \Big(\supp \nu_t \cap \bigcup_{i=0}^{k-1}A_i'\Big) +
   \mu \Big(  \supp \nu_t      \setminus \bigcup_{i=0}^{k-1}A_i'\Big)\Big\}
   \tag*{}\\
   = \ & \log \mu (\supp \nu_t) \tag*{}
   \end{align}Note that $(\nu_s)_{s\in [\,0,1\,]}$ is a geodesic between
  $\nu_0$ and $\nu_1$. Since
  \begin{align}\label{3ss14} \nu_i= \frac{(e_i)_{\ast}\Pi|_{\Gamma_t}}{\Pi(\Gamma_t)} \leq
 \frac{(e_i)_{\ast}\Pi}{\Pi(\Gamma_t)} =  \frac{\mu_i}{\Pi(\Gamma_t)}\leq
   \frac{a}{\Pi(\Gamma_t)} (\pr_{i+1})_{\ast}\pi=
   \frac{a}{\Pi(\Gamma_t)}\mu_{B_i}
   \end{align}for $i=0,1$, each $\nu_i$ is absolutely continuous with
  respect to $\mu$, and especially the above geodesic $(\nu_s)_{s\in
  [\,0,1\,]}$ is unique. For each $i=0,1$, we write $d \nu_i = \rho_i d\mu$. By (\ref{3s1}), we get
  \begin{align}\label{3s14}
   &\log \mu (\supp \nu_t )  \geq   - (1- t)
   \int_{X} \rho_{0}\log \rho_{0} d\mu - t  \int_{X}
   \rho_{1} \log \rho_{1} d\mu.
   \end{align}For a subset $A\subseteq X$ we denote by $1_A$ the
  characteristic function of $A$, i.e., $1_A(x):=1$ if $x\in A$ and
  $1_A(x):=0$ if $x\in X \setminus A$.
\begin{claim}\label{3cl4}We have
 \begin{align*}
  \rho_{i}\log \rho_{i}\leq \frac{c_t
  1_{B_i}}{\mu (B_i)}\log  \frac{c_t
  1_{B_i}}{\mu(B_i)} \ \ (i=0,1), 
  \end{align*}where $c_t:= a/\Pi(\Gamma_t)$.
 \begin{proof}By (\ref{3ss14}) we have $\rho_{i}\leq (c_t
  /\mu(B_i))1_{B_i}$. Since
  $c_t\geq 1$ and $u \log u \leq v \log v$ for any two positive
  numbers $u,v$ such
  that $u\leq v$ and $v\geq 1$, we obtain the claim.
  \end{proof}
 \end{claim}Combining Claim \ref{3cl4} with (\ref{3s15}) and (\ref{3s14}) we have
  \begin{align*}
   \log \mu (B_0) + \frac{\kappa^6}{\mu (B_0)} \geq \ & -(1-t) \int_{X}\frac{c_t
  1_{B_0}}{\mu(B_0)}\log  \frac{c_t
  1_{B_0}}{\mu (B_0)} d\mu -t \int_{X} \frac{c_t
  1_{B_1}}{\mu (B_1)}\log  \frac{c_t
  1_{B_1}}{\mu (B_1)} d\mu  \\
   =\ & -c_t \log c_t + c_t (1-t)\log
   \mu(B_0)  + c_t t \log \mu(B_1).
   \end{align*}Substituting $t:=\kappa^3$, we thereby obtain
  \begin{align}\label{3s16}
    \log (1/2) + 4\kappa^2 \geq \ &\log \frac{\mu(B_0)}{\mu(B_1)}+\frac{\kappa^6}{\kappa^3 \mu(B_0)} \\ \geq \ & -
   \frac{c_t}{\kappa^3}\log c_t + \frac{c_t -1
   }{\kappa^3} (1-\kappa^3)\log \mu (B_0 ) + (c_t-1 )\log \mu (B_1). \tag*{}
  \end{align}Using (\ref{3s11}) and (\ref{3s13}) we estimate each term on the
  right-side of the above inequalities as 
  \begin{align*}
   \frac{c_t\log c_t}{\kappa^3}  =\ &
   \frac{1}{(1-\kappa^6)(1-c\kappa^6)}\cdot \frac{-\log
   (1-\kappa^6)(1-c\kappa^6)}{\kappa^3} \leq
   \frac{2\kappa^3+2c\kappa^3}{(1-\kappa^6)(1-c\kappa^6)}
   \end{align*}
  \begin{align*}
       \Big|\frac{c_t -1 }{\kappa^3} \log \mu(B_0)\Big| \leq \ &\frac{a-\Pi(\Gamma_t)}{\kappa^3
   \Pi(\Gamma_t)}  \log \frac{2}{\kappa}\leq \frac{\kappa^3(
   \{1+c(1-\kappa^6)\}}{(1-\kappa^6)(1-c\kappa^6)}\log \frac{2}{\kappa}
   \end{align*}and
  \begin{align*}
   |(c_t-1)\log \mu(B_1)|\leq \frac{\kappa^6\{
   1+c(1-\kappa^6)\}}{(1-\kappa^6)(1-c\kappa^6)}\log \frac{1}{\kappa}.
   \end{align*}These estimates imply the right-side of the
  inequalities (\ref{3s16}) is close to zero for sufficiently small
  $\kappa>0$. Since the left-side of the inequality (\ref{3s16}) is about $\log(1/2)<0$ for
  sufficiently small $\kappa>0$,
  this is a contradiction. This
  completes the proof of the theorem.
  \end{proof}

\section{Conjecture}

We raise the following conjecture for eigenvalues of Laplacian.
 \begin{conj}\label{conjconj}If $(X,\mu)$ is an weighted compact finite-dimensional Alexandrov space of
  $\cd(0,\infty)$ and $k$ is a natural number, then we have
  \begin{align*}
   \lambda_{k+1}(X,\mu)\leq c \lambda_k(X,\mu)
   \end{align*}for some universal constant $c>0$.
  \end{conj}

  The answer is positive for any compact Riemannian homogeneous manifolds
  (\cite{chengyang}, \cite{li}).

  To explain how Theorem \ref{Mthm} relates with the above conjecture we
    need to explain some basics on the theory of concentration of measure
    in the sense of L\'evy and V.~Milman (\cite{levy}, \cite{mil2}). Refer to \cite{ledoux} for details.


We denote the open $r$-neighborhood of a subset $A$ in a metric space by $O_r(A)$.
\begin{dfn}[{Concentration function,~\cite{amil}}]\label{intdi}
  Let $X$ be an mm-space. For $r>0$ we define the real number $\alpha_X(r)$ as the
  supremum of $\mu_X ( X\setminus  O_r(A) )$, where $A$ runs over all
  Borel subsets of $X$ such that $\mu_X(A)\geq 1/2$. The function
  $\alpha_X:(\,0,+\infty\,)\to \mathbb{R}$ is called
  the \emph{concentration function}.
\end{dfn}


The following lemma asserts that exponential concentration
    inequalities and logarithmic $1$-separation inequalities are equivalent:

\begin{lem}\label{2.2l3}Let $X$ be an mm-space. 
 \begin{enumerate}
            \item If $X$ satisfies
                  \begin{align}\label{2.2s1}
                   \sep (X;\kappa,\kappa)\leq \frac{1}{C}\log \frac{c}{\kappa}\end{align}for any $\kappa>0$,
                  then we have $\alpha_X(r)\leq c \exp (-C r)$ for
                  any $r>0$.
           \item Conversely, if $X$ satisfies $\alpha_X(r)\leq c'\exp(-C' r)$ for any
                 $r>0$, then we have
                 \begin{align*}
                  \sep (X;\kappa,\kappa )\leq \frac{2}{C'}\log \frac{c'}{\kappa}
                  \end{align*}for any $\kappa>0$.
 \end{enumerate}
 \begin{proof}(1) Assume that $X$ satisfies (\ref{2.2s1}) and let $A\subseteq
  X$ be a Borel subset such that $\mu_X(A)\geq 1/2$. For $r>0$ we put
  $\kappa:=\mu_X(X\setminus O_r(A))$. Since
  \begin{align*}
   r\leq \dist(X\setminus O_r(A), A)\leq \sep(X;\kappa,1/2)\leq
   \sep(X;\kappa,\kappa)\leq \frac{1}{C}\log \frac{c}{\kappa},
   \end{align*}we have $\kappa\leq c\exp (-Cr)$, which gives the
  conclusion of (1).

  (2) Assuming that $\alpha_X(r)\leq c'\exp(-C'r)$, we take two Borel subsets
  $A,B\subseteq X$ such that $\mu_X(A)\geq \kappa$, $\mu_X(B)\geq
  \kappa$, and $\dist(A,B)=\sep(X;\kappa,\kappa)$. Let $\tilde{r}$ be
  any positive number satisfying
  \begin{align*}
  \alpha_X(\tilde{r})\leq c'\exp(-C'\tilde{r})<\kappa,
   \end{align*}i.e.,
  \begin{align*}
   \tilde{r}>\frac{1}{C'}\log \frac{c'}{\kappa}.
   \end{align*}Since $\mu_X(A)\geq \kappa$, by \cite[Lemma 1.1]{ledoux}, we have
  \begin{align*}
   1-\mu_X(O_{2\tilde{r}}(A))\leq \alpha_X(\tilde{r})<\kappa.
   \end{align*}Hence we have
  \begin{align*}
   \mu_X(O_{2\tilde{r}}(A)\cap B)> (1-\kappa)+\kappa - 1=0,
   \end{align*}which yields $\sep(X;\kappa,\kappa)=\dist(A,B)\leq
  2\tilde{r}$. Letting $\tilde{r}\to C'^{-1}\log (c'/\kappa)$ we obtain (2).
  \end{proof}
 \end{lem}

In the series of works \cite{emil2,emil3,emil1}, E.~Milman proved that a uniform tail-decay of the
concentration function implies the linear isoperimetric inequality
(Cheeger's isoperimetric inequality) under assuming the non-negativity of Bakry-\'Emery
Ricci curvature. Note that the linear isoperimetric inequality always
implies an appropriate Poincar\'e inequality and thus a lower bound for
the first eigenvalue of the weighted Laplacian. The key ingredient of E.~Milman's approach to the above result is the
 concavity of isoperimetric profile under the assumption of the non-negativity of
 Bakry-\'Emery Ricci
 curvature, the fact based on the regularity theory of isoperimetric
 minimizers (see \cite[Appendix]{emil2}). See also \cite{ledoux} for the
 heat semigroup approach. In \cite{goz} Gozlan, Roberto, and Samson proved
 that any exponential concentration inequalities imply appropriate Poincar\'e inequalities
 under assuming $\cd(0,\infty)$. In other words if an mm-space $X$
 satisfying $\cd(0,\infty)$ enjoys a logarithmic $1$-separation inequality $\sep
 (X;\kappa,\kappa)\leq (1/D)\log (1/\kappa)$ then we have
 $\sqrt{\lambda_1(X)}\geq c D$, where $c$ is some universal constant and
 $\lambda_1(X)$ is the spectral gap. Especially combining Theorem \ref{Mthm} for
 $k=2$ and $l=1$ with this theorem yields the positive answer to
 Conjecture \ref{conjconj} for $k=1$. In general according to Theorem
 \ref{Mthm} in order to give an affirmative answer to Conjecture
 \ref{conjconj} it suffices to extend E.~Milman's theorem or
 more weakly Gozlan-Roberto-Samson's theorem in terms of
 $\lambda_k(X,\mu)$, i.e., any logarithmic $k$-separation inequalities
 imply appropriate estimates of the $k$-th eigenvalue $\lambda_k(X,\mu)$
 from below under assuming $\cd (0,\infty)$.

\bigbreak
\noindent
\bigbreak
\noindent
{\it Acknowledgments.}

The author would like to thank to Professor Kazuhiro Kuwae for
explaining me analysis of Alexandrov spaces and answering my several questions. The author is also grateful to
an anonymous referee for carefully reading this paper and giving me helpful suggestions.

\end{document}